\documentclass[11pt]{amsart}

\usepackage{amsmath,amssymb,latexsym,soul,cite}
\usepackage{xcolor,enumitem,graphicx}
\usepackage[colorlinks=true,urlcolor=blue,
citecolor=red,linkcolor=blue,linktocpage,pdfpagelabels,
bookmarksnumbered,bookmarksopen]{hyperref}
\usepackage[english]{babel}
\usepackage[T1]{fontenc}
\usepackage[utf8]{inputenc}

\usepackage[left=2.8cm,right=2.8cm,top=2.9cm,bottom=2.9cm]{geometry}

\numberwithin{equation}{section}

\newtheorem{thm}{Theorem}[section]

  \theoremstyle{plain}
  \newtheorem{lem}[thm]{Lemma}
  \theoremstyle{plain}
  \newtheorem{prop}[thm]{Proposition}
  \theoremstyle{plain}
  
  \theoremstyle{plain}
  \newtheorem{definition}[thm]{Definition}
    \theoremstyle{definition}

\newcommand{\R}{{\mathbb R}}

\title[Uniqueness of least-energy solutions in the ball]{Uniqueness of least-energy solutions to the fractional Lane-Emden equation in the ball}

\author[A. DelaTorre]{Azahara DelaTorre}
\author[E. Parini]{Enea Parini}

\address[E. Parini]{Aix Marseille Univ, CNRS, I2M, 3 place Victor Hugo, 13331 Marseille CEDEX 03, France}
\email{enea.parini@univ-amu.fr}

\address[A. DelaTorre]{Dipartimento di Matematica Guido Castelnuovo. 
	Facoltà Scienze matematiche, fisiche e naturali.
	Sapienza Università di Roma. 
	Piazzale Aldo Moro, 5, 00185 Roma RM.} \email{azahara.delatorrepedraza@uniroma1.it}

\subjclass[2010]{}
\date{\today}

\thanks{}

\begin{document}

\begin{abstract}
We prove uniqueness of least-energy solutions to the fractional Lane-Emden equation, under homogeneous Dirichlet exterior conditions, when the underlying domain is a ball $B \subset \mathbb{R}^N$. The equation is characterized by a superlinear, subcritical power-like nonlinearity. The proof makes use of Morse theory and is inspired by some results obtained by C. S. Lin in the '90s. A new Hopf's Lemma-type result shown in this paper is an essential element in the proof of nondegeneracy of least-energy solutions.
\end{abstract}

\maketitle

\section{Introduction}

Let $\Omega \subset \R^N$ be a bounded, open set. The \emph{fractional Lane-Emden equation} reads as follows:

\begin{equation}\label{eq:main}
	\left\{\begin{array}{r c l l}
		(-\Delta)^s u & = & u^p & \text{in }\Omega, \\ u & = & 0 & \text{in }\R^N \setminus \Omega.
	\end{array}\right.
\end{equation}
Here $s \in (0,1)$, $p \in \left(1,\frac{N+2s}{N-2s}\right)$, and $(-\Delta)^s$ is the fractional Laplacian, defined, up to a positive multiplicative constant, as
\[ (-\Delta)^s u(x) := 
\int_{\R^N} \frac{u(x)-u(y)}{|x-y|^{N+2s}}\,dy, 
.\]
In the limit $s \to 1^-$, one recovers, after a suitable scaling, the usual Laplacian operator, and \eqref{eq:main} becomes
\begin{equation}\label{eq:mainlocal}
	\left\{\begin{array}{r c l l}
		-\Delta u &= & u^p & \text{in }\Omega, \\ u &= & 0 & \text{on }\partial \Omega.
	\end{array}\right.
\end{equation}
While existence of positive solutions to \eqref{eq:mainlocal} easily follows from standard variational methods, the question of their uniqueness is much more involved. In fact, uniqueness of positive solutions does not hold, in general, for annular domains \cite{dancer}. When $\Omega$ is a ball, however, a classical result by Gidas, Ni and Nirenberg \cite{GNN} states that positive solutions of \eqref{eq:mainlocal} are unique. The proof is based on the moving plane method, which guarantees that all positive solutions are radially symmetric, and which therefore allows to apply ODE techniques.

It is natural to wonder if these results carry over to the fractional case. When $\Omega$ is a ball, an extension of the moving plane method still allows to prove that all positive solutions to \eqref{eq:main} are radially symmetric (see \cite{chenlili}). However, it does not seem possible to use ODE arguments in nonlocal problems such as \eqref{eq:main}. A partial result has been obtained by Dieb, Ianni and Salda\~{n}a in \cite{diebiannisaldana} where, by perturbative methods, they obtained uniqueness of positive solutions for either $p$ or $s$ sufficiently close to $1$. Another partial result in the same direction have been obtained by the same authors in \cite{diebiannisaldana2}. In the very recent preprint \cite{fallwethonedim}, which was uploaded on ArXiv while our paper was being finalized, Fall and Weth proved uniqueness of positive solutions to \eqref{eq:main} in the one-dimensional case. We also mention the paper \cite{chandelmarhuangmaininivolzone}, where the authors obtained uniqueness of positive solutions to an equation in $\R^N$, driven by the fractional Laplacian, with convex nonlinearity. Unfortunately, we were not able to extend their interesting results to equation \eqref{eq:main}, since the fact of setting the equation in a ball instead of $\R^N$ seems to create a major technical difficulty.

For these reasons, in this paper we focus on \emph{least-energy} solutions, also called \emph{ground states}. By definition, these are the solutions of the minimization problem
\begin{equation*}
\inf\left\{\int_{\R^N} \int_{\R^N} \frac{|\phi(x)-\phi(y)|^2}{|x-y|^{N+2s}}\,dx\,dy \,\bigg|\,\phi\in H_0^s(\Omega) \text{ with }\int_\Omega |\phi|^{p+1}=1\right\},
\end{equation*}
whose existence follows from classical variational methods. It is straightforward to verify that least-energy solutions do not change sign in $\Omega$ and hence, by the strong maximum principle, are strictly positive (or negative) in $\Omega$. In particular, if $\Omega$ is a ball, least-energy solutions are radially symmetric. 

In the local case, Lin proved in \cite{lin}, by making use of Morse theory, that least-energy solutions to \eqref{eq:mainlocal} are unique whenever $\Omega$ is a convex, planar domain. We are able to extend his method of proof to the nonlocal case, and obtain the following result when $\Omega$ is a ball of any dimension $N\geq 2$:

\begin{thm} \label{thm:maintheorem}
Let $N \geq 2$, let $B \subset \R^N$ be a ball, $s \in (0,1)$, and $p \in \left(1, \frac{N+2s}{N-2s}\right)$. Then, the equation
\begin{equation}\label{eq:mainball}
	\left\{\begin{array}{r c l l}
		(-\Delta)^s u &= & u^p & \text{in }B, \\ u &= & 0 & \text{in }\R^N \setminus B,
	\end{array}\right.
\end{equation}
admits a unique (up to multiplication by $-1$) least-energy solution $u \in H^s_0(B)$, which is strictly positive (resp. strictly negative) in $B$, and radially symmetric. Moreover, $u$ is non-degenerate, and its Morse index is equal to one.
\end{thm}

 While Lin's ideas can be carried over without major issues in order to show that the Morse index of least-energy solutions is \emph{at least} one, the difficult part is to show that they are non-degenerate. Indeed, one is naturally led to the study of the linearized problem
\begin{equation}\label{eq:mainlinearized}
	\left\{\begin{array}{r c l l}
		(-\Delta)^s v - p u^{p-1}v &=& \mu v & \text{in }B, \\ v &=& 0 & \text{in }\R^N \setminus B,
	\end{array}\right.
\end{equation}
and precise information on the structure of second eigenfunctions is needed. Unlike in the local case, Courant's nodal domain theorem is not available, and neither is Hopf's Lemma for general sign-changing solutions (we mention however the recent results by Dipierro, Soave and Valdinoci \cite{dipierrosoavevaldinoci} which require a certain growth condition at the boundary). Even for the standard eigenvalue problem for the fractional Laplacian
\begin{equation}\label{eq:mainstandardeigenvalue}
\left\{\begin{array}{r c l l}
	(-\Delta)^s v  & = & \lambda v & \text{in }B, \\ v & = & 0 & \text{in }\R^N \setminus B,
\end{array}\right.
\end{equation}
it was shown only very recently, in \cite{fallfeulefacktemgouaweth} and \cite{benediktbobkovdharagirg}, that second eigenfunctions have exactly two nodal domains, and are antisymmetric, so that each nodal domain is a half-ball.

In the case under consideration in this paper, we are able to prove the following characterization of second eigenfunctions for problem \eqref{eq:mainlinearized}: 
\begin{itemize}
\item If $v$ is a radially symmetric second eigenfunction, then it changes  sign at most two times in the radial direction;
\item If there exists a nonradial second eigenfunction, then there exists an antisymmetric second eigenfunction $v$, with exactly two nodal domains which are half balls.
\end{itemize}	
In both cases, these eigenfunctions satisfy a fractional Hopf's Lemma. For antisymmetric eigenfunctions, this was proven in \cite{fallfeulefacktemgouaweth}, while the corresponding result for radially symmetric eigenfunctions is obtained in this paper. This allows to extend the result by Lin to the nonlocal case.

Let us mention that the very recent results by Fall and Weth \cite[Theorem 1.2]{fallwethonedim} and by Dieb, Ianni, and Salda\~{n}a \cite[Corollary 1.3]{diebiannisaldana2} actually rule out the possibility for \eqref{eq:mainlinearized} to admit antisymmetric second eigenfunctions. Nevertheless, we intend to discuss this case in our exposition as well, since we believe that it might be useful in some other context. 

A comment on the range of the exponent $p \in \left(1, \frac{N+2s}{N-2s}\right)$ in Theorem \ref{thm:maintheorem} is in order. If $p \in \left[ \frac{N+2s}{N-2s}, +\infty\right)$, equation \eqref{eq:mainball} does not admit any nontrivial positive solution, due to the fractional Pohozaev identity (see \cite[Corollary 1.3]{rosotonserra}) and Hopf's Lemma for positive solutions (see \cite[Lemma 3.1]{grecoservadei}).  If $p \in (0,1)$, the nonlinearity becomes \emph{sublinear}, and uniqueness of positive solutions holds true for \emph{every} domain $\Omega \subset \R^N$, as shown in \cite[Theorem 6.1]{brascofranzina}; the proof consists in showing that all positive solutions are in fact least-energy solutions and, subsequently, obtaining uniqueness of least-energy solutions by exploiting a particular convexity property of the energy functional, which the authors call \emph{hidden convexity}. Finally, the case $p=1$ is related to the eigenvalue problem for the fractional Laplacian, for which uniqueness - up to multiplicative constants - of the first eigenfunction follows from classical spectral theory (see, for instance, \cite[Theorem 2.8]{brascoparini}). 

The paper is structured as follows. After recalling the functional setting and the basic notions about least-energy solutions, in Section 3 we provide some basic results about the linearized problem and, in Section 4, we show the validity of the fractional Hopf's Lemma for the second eigenfunctions. Section 5 contains the proof of the nondegeneracy and uniqueness of least-energy solutions in the ball. Finally, in Section 6 we discuss some open questions.

\subsection*{Acknowledgements} The paper was initiated during a visit of A. DlT. to Marseille in May 2022, and continued during a visit of E. P. to Rome in January 2023. The authors express their gratitude to the hosting institutions for the hospitality that helped the development of this paper.

The authors would like to thank Mouhamed Moustapha Fall, Mar\'{i}a del Mar Gonz\'{a}lez, Mateusz Kwa\'{s}nicki, Xavier Ros-Oton, Giorgio Tortone, and Tobias Weth for valuable suggestions.

The results of this paper, as well as the main elements of the proofs, were first announced at the ``XII Workshop on Nonlinear Differential Equations'', held in Bras\'{\i}lia on September 11th-15th, 2023. Organizers and funding agencies are gratefully acknowledged.

A. DlT. acknowledges financial support from the Spanish Ministry of Science and Innovation (MICINN), through the IMAG-Maria de Maeztu Excellence Grant CEX2020-001105-M/AEI/ 10.13039/501100011033. She is also supported by the FEDER-MINECO Grants PID2021- 122122NB-I00 and PID2020-113596GB-I00; RED2022-134784-T, funded by MCIN/AEI/10.13039/ \newline 501100011033 and by J. Andalucia (FQM-116); Fondi Ateneo – Sapienza Università di Roma; PRIN (Prot. 20227HX33Z) and INdAM-GNAMPA Project 2023, codice CUP E53C2200193000 and INdAM -GNAMPA Project 2024, codice CUP E53C23001670001.

\section{Notations and preliminary results} \label{sec:preliminary}

\subsection{Functional setting}
Let $\Omega \subset \R^N$ be an open set with Lipschitz boundary, and let $s \in (0,1)$. We define the fractional Sobolev space $H^s_0(\Omega)$ as
\[ H^s_0(\Omega) := \{ u \in L^2(\R^N)\,|\,u = 0 \text{ in } \R^N \setminus \Omega,\,[u]_{H^s_0(\Omega)} < +\infty\},\]
where $[ \cdot ]_{H^s_0(\Omega)}$ is the \emph{Gagliardo seminorm} defined as
\[ [u]_{H^s_0(\Omega)} := \left( \int_{\R^N} \int_{\R^N} \frac{|u(x)-u(y)|^2}{|x-y|^{N+2s}}\,dx\,dy\right)^{\frac{1}{2}}.\]
It turns out that $u \mapsto [u]_{H^s_0(\Omega)}$ is a norm on $H^s_0(\Omega)$, which is equivalent to the full norm
\[ u \mapsto \|u\|_{L^2(\Omega)} + [u]_{H^s_0(\Omega)}\]
(see for instance \cite[Remark 2.5]{brascolindgrenparini}). $H^s_0(\Omega)$ is also a Hilbert space when endowed with the scalar product
\[ \langle u, v \rangle :=  \int_{\R^N} \int_{\R^N} \frac{(u(x)-u(y))(v(x)-v(y))}{|x-y|^{N+2s}}\,dx\,dy.\]
Moreover, the embedding
\[ H^s_0(\Omega) \hookrightarrow L^{p+1}(\Omega)\]
is compact for every $p \in \left(1,\frac{N+2s}{N-2s}\right)$ (see for instance \cite[Corollary 2.8]{brascolindgrenparini}).

\subsection{Least-energy solutions}
Let $\Omega \subset \R^N$ be an open set with Lipschitz boundary. The \emph{first eigenvalue} $\lambda_1(\Omega)$ of the fractional Laplacian under homogeneous Dirichlet conditions is the smallest $\lambda \in \R$ such that the eigenvalue problem
\begin{equation*} 
	\left\{\begin{array}{r c l l}
		(-\Delta)^s u &= & \lambda u  & \text{in }\Omega, \\ u &= & 0 & \text{in }\R^N \setminus \Omega,
	\end{array}\right.
\end{equation*}
admits a nontrivial solution in $H^s_0(\Omega)$. $\lambda_1(\Omega)$ can be characterized variationally as
\[ \lambda_1(\Omega) = \inf_{u \in H^s_0(\Omega) \setminus \{0\}} \frac{\displaystyle  \int_{\R^N} \int_{\R^N} \frac{|u(x)-u(y)|^2}{|x-y|^{N+2s}}\,dx\,dy}{\displaystyle \int_\Omega u^2}.\]
A \emph{least energy solution} of the equation
\begin{equation} \label{eq:mainpreliminary}
	\left\{\begin{array}{r c l l}
		(-\Delta)^s u &= & u^p & \text{in }\Omega, \\ u &= & 0 & \text{in }\R^N \setminus \Omega,
	\end{array}\right.
\end{equation}
is a solution obtained by minimization on $H^s_0(\Omega) \setminus \{0\}$ of the functional $\Phi :H^s_0(\Omega)\to \R$ defined as
\[ \Phi(u) = \int_{\R^N} \int_{\R^N} \frac{|u(x)-u(y)|^2}{|x-y|^{N+2s}}\,dx\,dy ,\]
under the constraint
\[ \int_\Omega |u|^{p+1} = 1.\]
The existence of a minimizer follows by the direct method of the Calculus of Variations, together with the compactness of the embedding $H^s_0(\Omega) \hookrightarrow L^{p+1}(\Omega)$. By generalizing \cite[Proposition 3.1]{franzinalicheri}, one can prove that minimizers are bounded in $\Omega$, and therefore belong to $C^s(\overline{\Omega})$ by \cite[Proposition 7.2]{rosoton}. Moreover, assuming some regularity on the boundary of $\Omega$, the function
\[ x \mapsto \frac{u(x)}{d(x)^s},\]
where $u$ solves \eqref{eq:mainpreliminary}, and $d(x) = \text{dist}(x;\partial \Omega)$, is also H\"{o}lder-continuous up to the boundary (see \cite[Proposition 7.4]{rosoton}). This allows to define a suitable notion of normal derivative:
\[ \partial_n^s u(z) := \lim_{\Omega \ni x \to z} \frac{u(x)}{d(x)^s}.\]
Least energy solutions of \eqref{eq:mainpreliminary} do not change their sign in $\Omega$. To see this, we first observe that, if $u \in H^s_0(\Omega) \setminus \{0\}$ is a minimizer such that $u^+,\,u^- \not \equiv 0$, then it holds
\[ \Phi(|u|) < \Phi(u),\]
a contradiction. Therefore, without loss of generality, $u$ is nonnegative in $\Omega$. Moreover, by \eqref{eq:mainpreliminary} it holds $(-\Delta)^s u \geq 0$ in $\Omega$ so that, by the strong maximum principle (see \cite[Theorem 2.1]{grecoservadei}), $u > 0$ in $\Omega$.

In the particular case where $\Omega$ is a ball, a generalization of the moving plane principle to the fractional Laplacian \cite{chenlili} allows to conclude that (positive) least energy solutions are radially symmetric, and radially decreasing. 

\section{The linearized problem} \label{sec:linearized}

In the following, $\Omega \subset \R^N$ will be an open set with Lipschitz boundary, and $u$ will be a least-energy solution of \eqref{eq:mainpreliminary}, where $s \in (0,1)$, and $p \in \left(1, \frac{N+2s}{N-2s}\right)$. In order to apply Morse theory, we are led to investigate the eigenvalue problem for the linearized operator
\[ v \mapsto (-\Delta)^s v - p u^{p-1} v,\]
that is,
	\begin{equation} \label{eq:eigenvalueproblem} \left\{ \begin{array}{r c l l} (-\Delta)^s v - pu^{p-1}v  & = &  \mu v & \text{in }B \\ v & = & 0 & \text{in }\R^N \setminus B.\end{array} \right. \end{equation}
By standard spectral theory, this problem admits a sequence of eigenvalues
\[ \mu_1 \leq \mu_2 \leq \dots \leq \mu_k \leq ... \]
where the eigenvalues are repeated according to their multiplicity, and the corresponding eigenfunctions are critical points of the functional $J:H^s_0(\Omega) \setminus \{0\} \to \R$ defined as
\[ J(v) := \frac{\displaystyle \int_{\R^N} \int_{\R^N} \frac{|v(x)-v(y)|^2}{|x-y|^{N+2s}}\,dx\,dy - p \int_\Omega u^{p-1}v^2 }{\displaystyle \int_\Omega v^2}.\]
In particular, it holds
\[ \mu_1 = \inf_{v \in H^s_0(\Omega) \setminus \{0\}} J(v).\]

The \emph{Morse index} of $u$ is the number of negative eigenvalues of the linearized problem \eqref{eq:eigenvalueproblem}. In the following we will prove that $u$ has Morse index equal to one, and that it is \emph{non-degenerate}, which means that $0$ is not an eigenvalue of \eqref{eq:eigenvalueproblem}. This amounts to prove that $\mu_1 < 0$, and $\mu_2 > 0$.

\begin{prop} \label{prop:lambda1}
Let $\mu_1$ be the first eigenvalue of \eqref{eq:eigenvalueproblem}. Then, \[\mu_1 < 0.\] Moreover, there exists a unique first eigenfunction $\varphi \in H^s_0(\Omega)$ (up to a nonzero multiplicative constant), which does not change its sign in $\Omega$. In particular,
\[ \mu_1 < \mu_2.\]
If $\Omega$ is a ball, then any positive (resp. negative) first eigenfunction is radially symmetric, and radially decreasing (resp. radially increasing).
\end{prop}
\begin{proof}
	It holds
	\[ J(u) =  \frac{\displaystyle (1-p)\int_\Omega u^{p+1}}{\displaystyle \int_\Omega u^2} < 0,\]
	so that clearly $\mu_1 < 0$. By standard variational methods, it can be shown that the infimum is attained for a nontrivial eigenfunction $\varphi \in H^s_0(\Omega) \setminus \{0\}$. We now prove that this eigenfunction is positive. First, we show that it is nonnegative. If, by contradiction, $\varphi$ changes its sign in $\Omega$, so that $\varphi^+$, $\varphi^- \not\equiv 0$, it would hold
	\[ J(|\varphi|) < J(\varphi),\]
	a contradiction. By the Maximum Principle, all eigenfunctions associated to $\mu_1$ are strictly positive (or negative) in $\Omega$. Indeed, let $\varphi\geq 0$ satisfy 
		\begin{equation} \left\{ \begin{array}{r c l l} (-\Delta)^s \varphi -  pu^{p-1} \varphi  & = & \mu_1 \varphi& \text{in }B \\ \varphi & = & 0 & \text{in }\R^N \setminus B,\end{array} \right. \end{equation}  and suppose there exists $x_0\in B$ such that $\varphi(x_0)=0$. Then \[(-\Delta)^s \varphi(x_0)= \int_{\Omega}\frac{\varphi(y)}{|y-x_0|^{N+2s}}\, dy=0,\] which is a contradiction since $\varphi \not\equiv 0$.

		Finally, if $\mu_1$ was not simple, it would be possible to find two eigenfunctions $\varphi_1$ and $\varphi_2$ which are orthogonal in $L^2(\Omega)$, so that
	\[ \int_\Omega \varphi_1 \varphi_2 = 0,\]
	a contradiction to the strict positivity (or negativity) of first eigenfunctions.

Let us now suppose that $\Omega$ is a ball, and let $\varphi \in H^s_0(\Omega)$ be a positive first eigenfunction. Let $\varphi^* \in H^s_0(\Omega)$ be the Schwarz symmetrization of $\varphi$, as defined for instance in \cite[Definition 1.3.1]{kesavan}. Since Schwarz symmetrization decreases the Gagliardo seminorm (see \cite[Section 9.1]{almgrenlieb}), it holds
\[ \int_{\R^N} \int_{\R^N} \frac{|\varphi^*(x)-\varphi^*(y)|^2}{|x-y|^{N+2s}}\,dx\,dy \leq \int_{\R^N} \int_{\R^N} \frac{|\varphi(x)-\varphi(y)|^2}{|x-y|^{N+2s}}\,dx\,dy.\]
Moreover, by the Hardy-Littlewood inequality we have
\[ - \int_\Omega u^{p-1}(\varphi^*)^2 \leq - \int_\Omega u^{p-1}\varphi^2,\]
where we use the fact the $u$ is radially symmetric, and radially decreasing. Finally, it is well-known that Schwarz symmetrization preserves the $L^2$-norm. All in all, we obtain
\[ J(\varphi^*) \leq J(\varphi)\]
which implies, by uniqueness, that $\varphi=\varphi^*$. As a consequence, $\varphi$ is radially symmetric, and radially decreasing. 
\end{proof}

\begin{prop} \label{prop:lambda2}
	Let $\mu_2$ be the second eigenvalue of \eqref{eq:eigenvalueproblem}. Then, \[\mu_2 \geq 0.\]
\end{prop}
\begin{proof}
The proof goes as in \cite[Lemma 1]{lin}. Let $\tilde{u} \in H^s_0(\Omega)$ be defined as
\[ \tilde{u}(x) = m^{-\frac{1}{p-1}} \cdot u(x),\]
where
\begin{equation} \label{eq:rayleighfortheproof} m = \inf_{v \in H^s_0(\Omega) \setminus \{0\}} \frac{[v]_{H^s_0(\Omega)}^2 }{\|v\|_{L^{p+1}(\Omega)}^2}.\end{equation}
By definition of least energy solution, $\tilde u$ is a minimizer for \eqref{eq:rayleighfortheproof}. As a consequence, for a fixed $w \in H^s_0(\Omega)$, if one defines $\psi:\R \to \R$ as
\[ \psi(t) := \frac{[\tilde{u}+tw ]_{H^s_0(\Omega)}^2  }{\|\tilde{u}+tw\|_{L^{p+1}(\Omega)}^2},\]
it holds $\psi'(0)=0$, and $\psi''(0) \geq 0$. This condition translates into
\[  \int_{\R^N} \int_{\R^N} \frac{|w(x)-w(y)|^2}{|x-y|^{N+2s}}\,dx\,dy - pm\int_\Omega \tilde{u}^{p-1} w^2 + \frac{p+1}{2}\left(\int_\Omega \tilde{u}^p w\right)^2 \geq 0\]
that is,
\begin{equation} \label{eq:inequalityforsecondeigenvalue}  \int_{\R^N} \int_{\R^N} \frac{|w(x)-w(y)|^2}{|x-y|^{N+2s}}\,dx\,dy  - p\int_\Omega u^{p-1} w^2 + m^{-\frac{2p}{p-1}}\cdot \frac{p+1}{2}\left(\int_\Omega u^p w\right)^2 \geq 0.\end{equation}
Using the minimax principle for $\mu_2$ \cite[Theorem 12.1]{LiebLoss}, which states that
\begin{equation} \label{eq:minimax} \mu_2 = \sup_{E \in \mathcal{E}_1} \inf_{v \in E^\perp} \frac{\displaystyle  \int_{\R^N} \int_{\R^N} \frac{|v(x)-v(y)|^2}{|x-y|^{N+2s}}\,dx\,dy - p\int_\Omega u^{p-1} v^2 }{\displaystyle \int_\Omega  v^2},\end{equation}
where $\mathcal{E}_1$ is the set of all $1$-dimensional linear subspaces of $H^s_0(\Omega)$, we obtain from \eqref{eq:inequalityforsecondeigenvalue}
\[ \mu_2 \geq 0\]
by choosing, in relation \eqref{eq:minimax}, \[ E = \left\{ w \in H^s_0(\Omega)\,\bigg |\, \int_\Omega u^p w = 0 \right\}.\]
\end{proof}

In the following, in order to show nondegeneracy of $u$, we will prove that $\mu_2$ is actually strictly positive. To this aim, we need to understand the structure of second eigenfunctions more in detail. Following \cite{benediktbobkovdharagirg}, we prove the following variational principle.

\begin{prop} \label{prop:analogtobobkov}
Let $v \in H^s_0(\Omega)$ be such that $v^+$, $v^- \not \equiv 0$, and 
\[ \mu_2 \int_\Omega (v^+)^2 \geq \langle v, v^+ \rangle - p\int_\Omega u^{p-1} (v^+)^2,\] \[  \mu_2 \int_\Omega (v^-)^2 \geq - \langle v, v^- \rangle - p\int_\Omega u^{p-1} (v^-)^2,\]	
where $\langle \cdot, \cdot \rangle$ is the scalar product on $H^s_0(\Omega)$ as defined in Section \ref{sec:preliminary}. Then, $v$ is a second eigenfunction of \eqref{eq:eigenvalueproblem}.
\end{prop}	

\begin{proof}
The proof follows the same lines of \cite[Lemma 2.1]{benediktbobkovdharagirg}.
\end{proof}

We now introduce the definition of \emph{polarization}, a symmetrization technique which turns out to be particularly useful for our purpose.

\begin{definition}
Let $v : \R^N \to \R$, and $a \in \R$. Let \[ H_a := \{ x \in \R^N\,|\,x_1 = a\},\] \[ \Sigma_a^+ := \{x \in \R^N\,|\,x_1 \geq a\},\quad \Sigma_a^- := \{x \in \R^N\,|\,x_1 \leq a\}.\] For $x = (x_1,\dots,x_N) \in \R^N$, denote by $\overline{x} = (2a-x_1,\dots,x_N) \in \R^N$ its reflection with respect to $H_a$. The \emph{polarization} of $v$ with respect to $H_a$ is the function $P_a v : \R^N \to \R$ defined as
\[ (P_a v)(x) := \left\{ \begin{array}{c l} \min\{v(x),v(\overline{x})\} & \text{if }x \in \Sigma_a^+, \\ \max\{v(x),v(\overline{x})\} & \text{if }x \in \Sigma_a^-.
	\end{array}\right.
	\]
\end{definition}

\begin{lem} \label{lem:extensionofbobkov}
	Let $B \subset \R^N$ be a ball of radius $R>0$ centered at the origin, and let $v \in H^s_0(B)$ be a second eigenfunction of \eqref{eq:eigenvalueproblem}. Let $P v$ be the polarization of $v$ with respect to the hyperplane
	\[ H := \{ x \in \R^N\,|\,x_1 = 0\}.\]
	Then, $P v \in H^s_0(B)$, and $P v$ is a second eigenfunction of \eqref{eq:eigenvalueproblem} as well.
\end{lem}
\begin{proof}
	
	First of all, we observe that
	\[ \int_B (v^+)^2 = \int_B ((P v)^+)^2, \qquad \int_B (v^-)^2 = \int_B ((P v)^-)^2\]
	since polarization preserves the $L^2$ norm in $\R^N$ (see for instance \cite[Section 2.3]{benediktbobkovdharagirg}), and all the functions involved have support in $B$. In a similar fashion, due to the radial symmetry of $u$, it holds
	\[ \int_B u^{p-1} (v^+)^2 = \int_B u^{p-1} ((P v)^+)^2, \qquad \int_B u^{p-1} (v^-)^2 = \int_B u^{p-1} ((P v)^-)^2.\]
	The fact that $P v \in H^s_0(B)$, and that
	\[ \langle v, v^+ \rangle \geq \langle P v, (P v)^+ \rangle, \qquad   -\langle v, v^- \rangle \geq - \langle P v, (P v)^- \rangle, \]
	is proven in \cite[Lemma 2.3]{benediktbobkovdharagirg}. 
Thus, we can assert that
	\[ \langle v, v^+ \rangle - p \int_B u^{p-1} (v^+)^2 \geq \langle P v, (P v)^+ \rangle - p \int_B u^{p-1} ((P v)^+)^2,\]
		\[ -\langle v, v^- \rangle - p \int_B u^{p-1} (v^-)^2  \geq -\langle P v, (P v)^- \rangle - p \int_B u^{p-1} ((P v)^-)^2.\]
By Proposition \ref{prop:analogtobobkov}, $P v$ is a second eigenfunction of \eqref{eq:eigenvalueproblem}.
\end{proof}

\section{Hopf's Lemma for the second eigenfunction}\label{Sec:Hopf}
Since no general Hopf's Lemma is available in our setting we need to study two mutually exclusive cases separately. In the first case, we will treat the possibility in which there exists a nonradial second eigenfunction for the linearized problem \eqref{eq:eigenvalueproblem}. We will be able to construct an antisymmetric second eigenvalue for which the Hopf's Lemma for antisymmetric functions proven in \cite{fallfeulefacktemgouaweth}. In the second case, we will provide a new Hopf's Lemma for radially symmetric functions that will be used if no nonradial second eigenfunctions exist.

\subsection{Case 1: \eqref{eq:eigenvalueproblem} admits a nonradial second eigenfunction}
\begin{definition}
	Let $B \subset \R^N$ be a ball of radius $R>0$ centered at the origin, and let $v : B \to \R$. $v$ is \emph{foliated Schwarz symmetric} with respect to $p \in \mathbb{S}^{N-1}$ if:
	\begin{itemize}
		\item $v$ is axially symmetric with respect to $p\R$, so that
		\[ v(x) = \tilde{v}\left( |x|, \arccos\left( \frac{x}{|x|} \cdot p\right)\right)\]
		for some function $\tilde{v}:[0,R] \times [0,\pi] \to \R$.
		\item $\tilde{v}$ is nonincreasing with respect to the second variable.
	\end{itemize}
\end{definition}

\begin{prop} \label{prop:existenceoffoliated}
	Let $B \subset \R^N$ be a ball of radius $R>0$ centered at the origin. Suppose that the eigenvalue problem \eqref{eq:eigenvalueproblem} admits a nonradial second eigenfunction. Then, \eqref{eq:eigenvalueproblem} admits a nonradial, foliated Schwarz symmetric second eigenfunction.
\end{prop}
\begin{proof}
Let $v$ be a nonradial second eigenfunction of \eqref{eq:eigenvalueproblem}. Since $v$ is nonradial, there exist $x_1$, $x_2 \in B$ such that $|x_1|=|x_2|$ and $v(x_1)\neq v(x_2)$. Without loss of generality, we can suppose that the segment connecting $x_1$ and $x_2$ is parallel to $e_1$. Let $P v$ be the polarization with respect to the hyperplane $P:= \{ x \in \R^N \,|\, x_1 = 0\}$. By Lemma \ref{lem:extensionofbobkov} and Proposition \ref{prop:analogtobobkov}, $P v$ is a second eigenfunction. Moreover, $P v$ is nonradial by construction, since $(P v)(x_1) \neq (P v)(x_2)$. Moreover, by \cite[Theorem 1.1]{jarohs}, it is foliated Schwarz symmetric. 
\end{proof}

We are now ready to state the main result of this subsection.

\begin{prop} \label{prop:existenceofantisymmetric}
Let $B \subset \R^N$ be a ball of radius $R>0$ centered at the origin. Define
\[ H:= \{ x \in \R^N \,|\, x_1 = 0\}, \] \[ \mathring{\Sigma}^+ :=  \{ x \in \R^N \,|\, x_1 > 0\}, \quad \mathring{\Sigma}^- :=  \{ x \in \R^N \,|\, x_1 < 0\}.\]
Suppose that the eigenvalue problem
\begin{equation} \label{eq:probleminexistenceofantisymmetric} \left\{ \begin{array}{r c l l} (-\Delta)^s v - pu^{p-1}v & = & \mu v& \text{in }B \\ v & = & 0 & \text{in }\R^N \setminus B\end{array}\right.\end{equation}
admits a nonradial second eigenfunction. Then \eqref{eq:probleminexistenceofantisymmetric} admits a second eigenfunction $v \in H^s_0(B)$ such that:
\begin{itemize}
 \item[(i)] $v$ is antisymmetric with respect to $H$;
 \item[(ii)] $v < 0$ in $B \cap \mathring{\Sigma}^+$, and $v>0$ in $B \cap \mathring{\Sigma}^-$; 
 \item[(iii)] $\partial_n^s v > 0$ on $\partial B \cap \mathring{\Sigma}^+$, and $\partial_n^s v < 0$ on $\partial B \cap \mathring{\Sigma}^-$.
 \end{itemize}
\end{prop}
\begin{proof}
Let $w$ be the nonradial, foliated Schwarz symmetric second eigenfunction of \eqref{eq:probleminexistenceofantisymmetric}, whose existence is guaranteed by Proposition \ref{prop:existenceoffoliated}. Without loss of generality, we can suppose that it is foliated Schwarz symmetric with respect to $e_1 = (1,0,\dots,0)$.
Define the function
\[ v(x) = w(x)-w(\overline{x}),\]
where $\overline{x}$ is the symmetric point to $x$ with respect to $H$. The function $v$ satisfies \eqref{eq:probleminexistenceofantisymmetric}, it is antisymmetric with respect to $H$ and, by Schwarz foliated symmetry, $v \leq 0$ in $B \cap \mathring{\Sigma}^+$. 

Moreover, $v \not\equiv 0$. Indeed, suppose by contradiction that $v \equiv 0$. This would mean in particular that $w(x)=w(\overline{x})$ for all $x \in \R e_1 \cap B$, which would imply that $w$ is constant on $\{ |x| = c \}$ for all $c \in [0,R]$ due to the foliated Schwarz symmetry. In other words, $w$ would be radially symmetric, a contradiction.

By the maximum principle and Hopf's Lemma for antisymmetric solutions in \cite[Proposition 3.3, Corollary 3.4]{falljarohs}, we have that $v < 0$ in $B \cap \mathring{\Sigma}^+$ and $\partial_n^s v > 0$ on $\partial B \cap \mathring{\Sigma}^+$, the remaining claims following by antisymmetry. 
	
\end{proof}	

\subsection{Case 2: \eqref{eq:eigenvalueproblem} admits only radially symmetric second eigenfunctions}

In this subsection we will first identify a sufficient condition which guarantees the validity of a weak minimum principle. We will then show that radially symmetric second eigenfunctions to \eqref{eq:eigenvalueproblem}, even when slightly perturbed, satisfy indeed this condition. As a consequence, we will be able to show the validity of Hopf's Lemma in this setting.

\begin{prop} \label{prop:weakminimumprinciple}
	Let $B \subset \R^N$ be the unit ball. Let $c \in L^\infty(B)$, and let $h \in H^s_0(B)$. Suppose that there exists $r_0 \in (0,1)$ such that, for every $r \in (r_0,1)$, 
		\begin{equation} \label{eq:conditionforhopf}  \quad \int_{B \setminus A_r} \frac{h(x)}{|x-y|^{N+2s}}\,dx \geq 0 \quad \text{for every } y \in A_r,\end{equation}
		where $A_r := \{ x \in \R^N \,|\, r<|x|<1\}$ .
		Let $r \in (r_0,1)$ be such that $\lambda_1(A_r) > \|c\|_\infty$, where $\lambda_1(A_r)$ denotes the first eigenvalue of the Dirichlet fractional Laplacian in $A_r$, and let $w \in H^s(\R^N)$ satisfy
	\begin{equation}
		\left\{\begin{array}{r c l l}
			(-\Delta)^s w &\geq& c(x)w& \text{in }A_r, \\ w &= & h & \text{in }\R^N \setminus A_r.\end{array}\right.
	\end{equation}
Suppose that $w^-|_{A_r} \in H^s_0(A_r)$. Then $w \geq 0$ in $A_r$.
\end{prop}
\begin{proof}
	Set $w=w^+-w^-$. Testing the equation with $w^- |_{A_r}$ we obtain
	\begin{align*}
		-\int_{A_r} c(x)(w^-)^2 & \leq \int_{\R^N} \int_{\R^N} \frac{(w(x)-w(y))(w^-(x)|_{A_r}-w^-(y)|_{A_r})}{|x-y|^{N+2s}}\,dx\,dy = \\ & \leq -\int_{\R^N} \int_{\R^N} \frac{(w^-(x)|_{A_r}-w^-(y)|_{A_r})^2}{|x-y|^{N+2s}}\,dx\,dy - 2\int_{A_r} \int_{A_r} \frac{w^+(x)w^-(y)}{|x-y|^{N+2s}}\,dx\,dy \\ & - 2\int_{A_r} \int_{B \setminus A_r} \frac{w(x)w^-(y)}{|x-y|^{N+2s}}\,dx\,dy \\ & \leq -\lambda_1(A_r) \int_{A_r}(w^-)^2 - 2 \int_{A_r} w^-(y) \left(\int_{B \setminus A_r} \frac{h(x)}{|x-y|^{N+2s}}\,dx \right)\,dy,
	\end{align*}
Thus, by assumption \eqref{eq:conditionforhopf},
	\begin{equation}\label{1} (\lambda_1(A_r) - \|c\|_\infty )\int_{A_r}(w^-)^2 \leq 0,\end{equation}
which implies $w^- \equiv 0$ in $A_r$.

	\end{proof}

\begin{prop} \label{prop:valueoftheintegral}
Let $B \subset \R^N$ be the unit ball, and let $v \in H^s_0(B)$ be a radially symmetric second eigenfunction of \eqref{eq:eigenvalueproblem}. Then, for any $x_0 \in \partial B$, it holds
\[ \int_{B} \frac{v(y)}{|x_0-y|^{N+2s}}\,dy > 0,\]
where the value of the integral might be $+\infty$.
\end{prop}	
\begin{proof}		
	 Let $\tilde{v}:[0,1] \to \R$ be such that $v(x)=\tilde{v}(|x|)$. Then, by reasoning as in \cite[Proposition 5.4]{franklenzmannsilvestre}, $\tilde{v}$ changes sign at most two times in $[0,1]$. Without loss of generality, we can suppose that $\tilde{v}$ is nonnegative in a left neighbourhood of $1$. Therefore, we have to consider two cases.
	
	 {\bf Case (i):} There exists $R_1 \in (0,1)$ such that $\tilde{v} \leq 0$ in $[0,R_1]$, and $\tilde{v} \geq 0$ in $[R_1,1]$. 
	 
	Let $x_0\in\partial B$. By radial symmetry, it holds
	 \begin{align*}
	 	\int_B \frac{v(y)}{|x_0-y|^{N+2s}}\,dy & = \frac{1}{\mathcal{H}^{N-1}(\partial B)} \int_{\partial B} \left( \int_B \frac{v(y)}{|x-y|^{N+2s}}\,dy \right)\,d\mathcal{H}^{N-1}(x) \\ & = \frac{1}{\mathcal{H}^{N-1}(\partial B)} \int_{B} v(y) \left( \int_{\partial B} \frac{1}{|x-y|^{N+2s}}\,d\mathcal{H}^{N-1}(x) \right)\,dy \\
	 	& = \frac{1}{\mathcal{H}^{N-1}(\partial B)} \int_{B} v(y) g(y)\,dy,
	 \end{align*}
	 where
	 \[ g(y) :=  \int_{\partial B} \frac{1}{|x-y|^{N+2s}}\,d\mathcal{H}^{N-1}(x), \]
	 and $\mathcal{H}^{N-1}(\partial B)$ is the $(N-1)-$dimensional Hausdorff measure of $\partial B$. We observe that $g$ is radially symmetric. Moreover, using some properties of the Gauss hypergeometric function defined by  
	 \begin{equation} \label{eq:hypergeom} {}_2F_1(a,b;c;z):=\frac{\Gamma\left(c\right)}{\Gamma\left(a\right)\Gamma\left(b\right)}\sum_{k=0}^{\infty}\frac{\Gamma\left(a+k\right)\Gamma\left(b+k\right)}{\Gamma\left(c+k\right)}\frac{z^k}{k!},\end{equation}
	 we can rewrite (see for instance \cite[Section 2.5.1]{Magnus}, see also \cite[Section 2.2]{DDGW} and the references therein)
	 \begin{align*}  g(y)&=\int_{\partial B} \frac{1}{|x-y|^{N+2s}}\,d\mathcal{H}^{N-1}(x)=\int_{\partial B} \frac{1}{|1+|y|^2-2|y|\langle\theta,\sigma\rangle|^{\frac{N+2s}{2}}}\,d\mathcal{H}^{N-1}(x) \\ &=c_N\int_0^\pi \frac{\sin^{N-2} \phi_1}{|1+|y|^2-2|y|\cos\phi_1|^{\frac{N+2s}{2}}}\,d\phi_1=c_N\frac{\sqrt{\pi}\Gamma\left(\frac{N-1}{2}\right)}{\Gamma\left(\frac{N}{2}\right)}{}_2F_1\left(\frac{N+2s}{2},1+s;\frac{N}{2}; |y|^2\right),
	 \end{align*}
	 where we have used the rotational invariance of the integral in the variable $\theta$, and $c_N$ is a positive constant depending only on the integral on the other spherical coordinates. 
	
Since the coefficients of the series in \eqref{eq:hypergeom} are positive,  $g$ is radially increasing. Furthermore, it is known \cite[15.4(ii)]{Olde} that as $|y|\to1^-$,
	 \[ {}_2F_1\left(\frac{N+2s}{2},1+s;\frac{N}{2}; |y|^2\right) \sim \frac{\Gamma\left(\frac{N}{2}\right)\Gamma\left( 1+2s\right)}{\Gamma\left(\frac{N}{2}+s\right)\Gamma\left( 1+s\right)} (1-|y|^2)^{-1-2s},\]
	 and therefore

	 \[ g(0)=c_N\frac{\sqrt{\pi}\Gamma\left(\frac{N-1}{2}\right)}{\Gamma\left(\frac{N}{2}\right)}>0\quad 
	 \text{ and }\lim_{|y|\to 1^-} g(y)=+\infty.
	 \]
	 Let $\varphi_1 \in H^s_0(B)$ be a positive first eigenfunction of \eqref{eq:eigenvalueproblem}. By Proposition \ref{prop:lambda1}, $\varphi_1$ is radially symmetric, and radially decreasing. Let $\hat{x} \in B$ be such that $|\hat{x}|=R_1$, which implies that $v(\hat{x}) = 0$. The function
	 \[ y \mapsto g(y)-\frac{g(\hat{x})}{\varphi_1(\hat{x})}\varphi_1(y)\]
	 is radially symmetric, and radially increasing. Moreover, it is negative for $|y| < R_1$, and positive for $|y|>R_1$. As a consequence,
	 \begin{align*}
	 	\int_B \frac{v(y)}{|x_0-y|^{N+2s}}\,dy  &=	\int_B v(y)g(y)\,dy \\ & =\int_B \left[v^+(y)\left(g(y)-\frac{g(\hat{x})}{\varphi_1(\hat{x})}\varphi_1(y)\right)-v^-(y)\left(g(y)-\frac{g(\hat{x})}{\varphi_1(\hat{x})}\varphi_1(y)\right)\right]\,dy\\
	 	&=\int_{\{R_1<|y|<1\}} v^+(y)\underbrace{\left(g(y)-\frac{g(\hat{x})}{\varphi_1(\hat{x})}\varphi_1(y)\right)}_{> 0}-\int_{\{|y| < R_1\}}v^-(y)\underbrace{\left(g(y)-\frac{g(\hat{x})}{\varphi_1(\hat{x})}\varphi_1(y)\right)}_{<0}\,dy\\ &>0.
	 \end{align*}

	 {\bf Case (ii):} There exist $R_1,\,R_2 \in (0,1)$ with $R_1 < R_2$, such that $\tilde{v} \leq 0$ in $[R_1,R_2]$, and $\tilde{v} \geq 0$ in $[0,R_1] \cup [R_2,1]$. 

	 Let $V \in (C \cap L^{\infty})(\R^{N+1}_+)$ be the Caffarelli-Silvestre extension of $v$ in $\R^{N+1}_+:=\R^N \times \R^+$, namely, the solution of
	 \begin{equation} \label{eq:caffarellisilvestre}
		\left\{\begin{array}{r c l l}
		\text{div}(t^{1-2s} \nabla V) &=& 0& \text{in }\R^{N+1}_+, \\ V(x,0) &= & v(x) & \text{in }\R^N,\end{array}\right.
\end{equation}
where $(x,t) \in \R^N \times \R^+$, and $v$ is extended by zero outside of $B$. Since $v$ is radially symmetric, $V$ enjoys cylindrical symmetry. By \cite[Proposition 5.4]{franklenzmannsilvestre}, $V$ has exactly two nodal domains \[ \mathcal{N}^+ := \left\{(x,t) \in \R^N \times \R^+ \,|\,V(x,t) >0\right\},\quad \mathcal{N}^- := \left\{(x,t) \in \R^N \times \R^+ \,|\,V(x,t) <0\right\}.\]
Let $e_1 := (1,0,\dots,0) \in \R^N$, and let $x_1$, $x_2 \in \R e_1$ be such that $|x_1|<R_1$, $R_2<|x_2|<1$, and $v(x_1)$, $v(x_2) >0$. Such points exist since, by the unique continuation property for the fractional Laplacian as stated in \cite[Theorem 1.2]{ghoshsalouhlmann}, $v$ can not be identically zero on an open set. Let $x_0 = e_1 \in \partial B$. 
Since $\mathcal{N}^+$ is connected, there exists a continuous path $\gamma:[0,1] \to \R^{N+1}_+$ contained in $\mathcal{N}^+ $ such that $\gamma(0) = x_1$, $\gamma(1)=x_2$. By cylindrical symmetry of $V$, we can suppose without loss of generality that $\gamma((0,1)) \subset \mathcal{N}^+ \cap (\R^+ e_1 \times \R^+)$.
This means that $\mathcal{N}^-$ is enclosed between the surface generated by the rotation of $\gamma$ around the positive semiaxis, and the hyperplane $\R^N$. Since $\gamma([0,1])$ is compact, and $x_0 \not \in \gamma([0,1])$, the distance between $x_0$ and $\gamma([0,1])$ is positive, and so is the distance between $x_0$ and $\mathcal{N}^-$. In particular, there exists a neighbourhood $W \subset \R^{N+1}_+$ such that $V\geq 0$ in $ \mathring{W}$. The strong maximum principle for strictly elliptic operators guarantees that either $V\equiv 0$ or  $V>0$ in $ \mathring{W}$ (see for example \cite[Remark 4.2]{cabresire}). The former case cannot occur because otherwise by continuity we would have $v=0$ in an open set, a contradiction to the unique continuation property for the fractional Laplacian \cite[Theorem 1.2]{ghoshsalouhlmann}. Since $V(x_0,0)=v(x_0)=0$, the Hopf's Lemma for the Caffarelli-Silvestre extension \cite[Proposition 4.11]{cabresire} implies
\begin{equation} \label{eq:hopfcaffsilv}  \liminf_{t \to 0^+} \frac{V(x_0,t)-V(x_0,0)}{t^{2s}} > 0.\end{equation}
Arguing as in \cite[Section 3.1]{caffarellisilvestre} we obtain, by using the change of variable $z= \left( \frac{t}{2s}\right)^{2s}$,
\[ \frac{V(x_0,z)-V(x_0,0)}{z} = C \int_{\R^N} \frac{v(y)}{\left(|x_0-y|^2 + 4s^2 |z|^2\right)^{\frac{N+2s}{2}}}\,dy\]
for every $z>0$, where $C$ is a constant depending on $N$ and $s$. Since $v$ is nonnegative in a neighbourhood of $\partial B$, there exists $\delta >0$ such that $v$ is nonnegative in $B_\delta(x_0)$. Let us decompose the integral into the sum of two terms:
\begin{align*}  \int_{\R^N} \frac{v(y)}{\left(|x_0-y|^2 + 4s^2 |z|^2\right)^{\frac{N+2s}{2}}}\,dy  & =  \int_{\R^N \setminus B_\delta(x_0)} \frac{v(y)}{\left(|x_0-y|^2 + 4s^2 |z|^2\right)^{\frac{N+2s}{2}}}\,dy \\ & +  \int_{B_\delta(x_0)} \frac{v(y)}{\left(|x_0-y|^2 + 4s^2 |z|^2\right)^{\frac{N+2s}{2}}}\,dy.\end{align*}
For $z \to 0^+$, the first integral converges to
\[ \int_{\R^N \setminus B_\delta(x_0)} \frac{v(y)}{|x_0-y|^{N+2s}}\,dy \]
by Lebesgue's Dominated Convergence Theorem, while the second integral tends to 
\[ \int_{B_\delta(x_0)} \frac{v(y)}{|x_0-y|^{N+2s}}\,dy \]
by the Monotone Convergence Theorem. Observe that this last quantity might be equal to $+\infty$.
By \eqref{eq:hopfcaffsilv} we obtain that
\[ \int_{\R^N} \frac{v(y)}{|x_0-y|^{N+2s}}\,dy  > 0,\]
where the value of the integral might be equal to $+\infty$.
\end{proof}

As a direct consequence, we obtain that any radial second eigenfunction of \eqref{eq:eigenvalueproblem} satisfies Hopf's Lemma.

\begin{prop} \label{prop:hopflemmaradial}
Let $B \subset \R^N$ be the unit ball, and let $v \in H^s_0(B)$ be a radially symmetric second eigenfunction of \eqref{eq:eigenvalueproblem}, such that $v \geq 0$ in a neighbourhood of $\partial B$. Then, for every $x_0 \in \partial B$,
\[\partial^s_n v(x_0)= \lim_{x \to x_0} \frac{v(x)}{d(x)^s} > 0.\]
\end{prop}
\begin{proof}
	The proof is inspired by \cite[Proposition 3.3]{falljarohs}. 
	By Proposition \ref{prop:valueoftheintegral}, for every $x_0 \in \partial B$,
	\[ \int_B \frac{v(y)}{|x_0-y|^{N+2s}}\,dy > 0.\]
	Therefore, there exists  $r_0 \in (0,1)$ such that, for every $r \in (r_0,1)$, 
	\[  \int_{B \setminus A_r} \frac{v(x)}{|x-y|^{N+2s}}\,dx > 0 \quad \text{for every } y \in A_r,\]
	where $A_r := \{x \in \R^N \,|\, r < |x|<1\}$. Moreover, by assumption, we can suppose without loss of generality that $v \geq 0$ in $A_{r_0}$. Set $c(x) = pu^{p-1} + \mu$. By the Faber-Krahn inequality, and the scaling properties of the eigenvalues of the fractional Laplacian, it holds
	\[ \lim_{r \to 1^-} \lambda_1(A_r) \to +\infty.\]
	Therefore, it is possible to choose $r \in (r_0,1)$ such that $\lambda_1(A_r)>\|c\|_\infty$. 
	Let $\psi \in H^s_0(A_r)$ be the solution of
	\begin{equation*}
		\left\{\begin{array}{r c l l}
			(-\Delta)^s \psi &=& 1& \text{in }A_r, \\ \psi &= & 0 & \text{in }\R^N \setminus A_r.\end{array}\right.
	\end{equation*}
By \cite[Lemma 1.2]{grecoservadei}, $\psi$ satisfies Hopf's Lemma at every point of $\partial B$.	Let $R \in (0,1)$, be such that the ball $B_{2R}$ of radius $2R$ centered at the origin is such that $B_{2R} \cap A_r = \emptyset$. For $\alpha > 0$, define
	\[ \eta := \psi + \alpha \xi,\]
	where $\xi \in C^\infty_c(B_{2R})$ is a nonnegative function, such that $\xi \geq 1$ in $B_R$. Let $\varphi \in H^s_0(A_r)$ be a nonnegative function. Then we have
	\begin{align*}
		\int_{\R^N} \int_{\R^N} \frac{(\eta(x)-\eta(y))(\varphi(x)-\varphi(y))}{|x-y|^{N+2s}}\,dx\,dy & = 	\int_{\R^N} \int_{\R^N} \frac{(\psi(x)-\psi(y))(\varphi(x)-\varphi(y))}{|x-y|^{N+2s}}\,dx\,dy \\ & + \alpha 	\int_{\R^N} \int_{\R^N} \frac{(\xi(x)-\xi(y))(\varphi(x)-\varphi(y))}{|x-y|^{N+2s}}\,dx\,dy \\ & = \int_{A_r} \varphi(x)\,dx - 2\alpha \int_{A_r} \int_{B_{2R}} \frac{\varphi(x)\xi(y)}{|x-y|^{N+2s}}\,dx\,dy \\ & \leq \int_{A_r} \varphi(x)\,dx - 2\alpha \int_{A_r} \int_{B_{R}} \frac{\varphi(x)}{|x-y|^{N+2s}}\,dx\,dy \\ & \leq (1-\alpha C) \int_{A_r} \varphi(x)\,dx.
	\end{align*}
	Here $C = C(R)$ (notice that $C$ depends on the \emph{maximal} distance between points of $A_r$ and $B_R$). By taking $\alpha$ sufficiently big so that
	\[ 1-\alpha C \leq  -\|c\|_\infty,\]
	we obtain that \[ (-\Delta)^s \eta \leq -\|c\|_\infty \eta \leq c(x) \eta \qquad \text{in } A_r\] since $\eta\geq 0$ in $A_r$.
	 By continuity, there exists $\varepsilon > 0$ sufficiently small such that the condition
	\[  \quad \int_{B \setminus A_r} \frac{v(x) - \varepsilon \eta(x)}{|x_0-y|^{N+2s}}\,dx >0 \quad \text{for every } y \in A_r\]
	is satisfied. Recalling that $v$ is non-negative in $A_r$, we can assert that $(v-\varepsilon \eta)^-|_{A_r} \in H^s_0(A_r)$. By Proposition \ref{prop:weakminimumprinciple} with $w=v-\varepsilon \eta \in H^s(\R^N)$ it holds
	\[ v \geq \varepsilon \psi \qquad \text{in }A_r,\]
	which implies that $v$ satisfies Hopf's Lemma on every point of $\partial B$.
\end{proof}

\section{Uniqueness of least-energy solutions}\label{Sec:Uniq}

The results of the previous section on the structure of second eigenfunctions of the linearized problem, and the validity of Hopf's Lemma, allow to prove nondegeneracy of least-energy solutions to \eqref{eq:mainball} and, as a consequence, their uniqueness.

\begin{prop} \label{prop:nondegeneracy}
Let $B \subset \R^N$, and let $u \in H^s_0(B)$ be a least-energy solution of \eqref{eq:mainball}. Then, $u$ is non-degenerate.
\end{prop}
\begin{proof}
Let $\mu_1$ and $\mu_2$ be the first two eigenvalues of the linearized problem \eqref{eq:mainlinearized}. We already proved that $\mu_1 <0$ (Proposition \ref{prop:lambda1}), and $\mu_2 \geq 0$ (Proposition \ref{prop:lambda2}). In order to show that $u$ is non-degenerate, it is enough to show that $\mu_2 > 0$.

Suppose by contradiction that $\mu_2=0$. We distinguish two cases.

{\bf Case 1: The linearized problem admits a nonradial second eigenfunction.}

By Proposition \ref{prop:existenceofantisymmetric}, there exists a second eigenfunction $v \in H^s_0(B)$ of the linearized problem, which solves
\[ \left\{ \begin{array}{r c l l} (-\Delta)^s v & = & p u^{p-1} v & \text{in }B, \\ v & = & 0 & \text{in }\R^N \setminus B,\end{array}\right.\]
and which satisfies conditions (i), (ii), (iii) in that proposition. Arguing as in \cite[Theorem 3.3]{brascolindgrenparini}, one can prove that $v \in L^\infty(B)$, and therefore $v \in C^s(\overline{B})$ by \cite[Proposition 7.2]{rosoton}.
Testing equation \eqref{eq:mainball} with $v$, and equation \eqref{eq:mainlinearized} with $u$, we obtain
\[ \int_B u^p v = p \int_B u^p v\]
which implies, since $p>1$,
\[ \int_B u^p v = 0.\]
Let $e_1:=(1,0,\dots,0)$. By applying \cite[Proposition 1.6]{rosotonserra} with center $e_1$ to the functions $u+v$ and $u-v$, and then taking the difference, we obtain a bilinear Pohozaev identity:
\begin{align} \label{bilinearpohozaev} & \int_B ((x-e_1) \cdot \nabla u) (-\Delta)^s v + \int_B ((x-e_1) \cdot \nabla v) (-\Delta)^s u  \\ \nonumber & = \left( s- \frac{N}{2}\right)\left[ \int_B u(-\Delta)^s v + \int_B v(-\Delta)^s u\right] - \Gamma(1+s)^2 \int_{\partial B} \partial^s_n u \,\partial^s_n v \,((x-e_1) \cdot \nu),\end{align}
where $\nu$ is the outward normal vector on $\partial B$. We point out that we are allowed to apply \cite[Proposition 1.6]{rosotonserra} since $u$ and $v$ are bounded solutions of equations of the form given in \cite[Equation (1.8)]{rosotonserra}, and therefore they satisfy the required assumptions thanks to \cite{rosotonserra2} (see also \cite[Theorem 1.4]{rosotonserra}).

The left-hand side of \eqref{bilinearpohozaev} is equal to
\[ p\int_B ((x-e_1) \cdot \nabla u) u^{p-1}v  + \int_B ((x-e_1) \cdot \nabla v) u^p  \]
which is equal to
\[ \int_B (x-e_1) \cdot \nabla (u^{p}) v + \int_B ((x-e_1) \cdot \nabla v) u^p \]
and, after integration by parts, to
\[ - N \int_B u^p v  = 0. \]
On the other hand, the right-hand side of \eqref{bilinearpohozaev} is equal to
\[\left( s- \frac{N}{2}\right)\left[ p\int_B u^p v+ \int_B u^p v\right] - \Gamma(1+s)^2 \int_{\partial B} \partial^s_n u \,\partial^s_n v \,((x-e_1) \cdot \nu)\]
and therefore it is equal to
\[  - \Gamma(1+s)^2 \int_{\partial B} \partial^s_n u \,\partial^s_n v \,((x-e_1) \cdot \nu).\]
All in all, we obtain that
\[\int_{\partial B} \partial^s_n u \,\partial^s_n v \,((x-e_1) \cdot \nu)= 0.\]
Since $u$ is nonnegative and radially symmetric, by the fractional Hopf's Lemma \cite[Lemma 3.1]{grecoservadei} we obtain $\partial^s_n u = \text{const.} < 0$ on $\partial B$, and hence
\[\int_{\partial B} \partial^s_n v \,((x-e_1) \cdot \nu)= 0.\]
Repeating the reasoning with $-e_1$ instead of $e_1$, we have
\[\int_{\partial B} \partial^s_n v \,((x+e_1) \cdot \nu)= 0\]
and, subtracting the two equations,
\[\int_{\partial B} \partial^s_n v \,(e_1 \cdot \nu)= 0.\]
Since $v$ satisfies property (iii) in Proposition \ref{prop:existenceofantisymmetric}, we reach a contradiction. Therefore, $\mu_2 > 0$, and $u$ is non-degenerate.

{\bf Case 2: The linearized problem admits only radially symmetric second eigenfunctions.}

Let $v \in H^s_0(B)$ be a radially symmetric second eigenfunction. Arguing as in the previous case, we obtain the bilinear Pohozaev identity
\begin{align} \label{bilinearpohozaevradial} & \int_B (x \cdot \nabla u) (-\Delta)^s v + \int_B (x \cdot \nabla v) (-\Delta)^s u  \\ \nonumber & = \left( s- \frac{N}{2}\right)\left[ \int_B u(-\Delta)^s v + \int_B v(-\Delta)^s u\right] - \Gamma(1+s)^2 \int_{\partial B} \partial^s_n u \,\partial^s_n v \,(x \cdot \nu).\end{align}
which eventually leads, by following the same steps as above, to the condition
\[\int_{\partial B} \partial^s_n v \,(x \cdot \nu)= 0,\]
which is the same as
\[\int_{\partial B} \partial^s_n v = 0,\]
a contradiction to the Hopf's Lemma proved in Proposition \ref{prop:hopflemmaradial}. Therefore, $\mu_2 > 0$, and $u$ is non-degenerate.

\end{proof}
We are now ready to prove the main result of this paper.

\begin{proof}[{\bf Proof of Theorem \ref{thm:maintheorem}}]

The proof is done by contradiction following the ideas of the local case \cite{lin}, and it is based on a continuation argument in combination with the Implicit Function Theorem. Indeed, the non degeneracy of solutions given by Proposition \ref{prop:nondegeneracy} implies the injectivity of the Fr\'{e}chet derivative of the operator  $F:H^s_0(B) \times \left(1,\frac{N+2s}{N-2s}\right) \to H^{-s}(B)$ defined as
\[F(u,p):=(-\Delta)^su-u^p .\]
 Thus, given $(u,p) \in \left(1,\frac{N+2s}{N-2s}\right)$ satisfying \eqref{eq:main}, the Implicit Function Theorem ensures the uniqueness of solution in a sufficiently small neighborhood of $u$ and $p$. Suppose there exists $\bar{p}\in\left(1, \frac{N+2s}{N-2s}\right)$ such that we have two distinct solutions $u_1(\bar p),u_2(\bar p)\in H^s_0(B)$; a continuation argument on $\bar{p}$ and the uniqueness of positive solutions (and hence, of least-energy solutions) for $p$ close to $1$ given by \cite{diebiannisaldana} will give a contradiction. 
\end{proof}

\section{Final remarks and open questions}

The results of Sections \ref{Sec:Hopf} and  \ref{Sec:Uniq} do not rely on the energy minimizing property of the solutions, and therefore they are actually valid for general solutions to \eqref{eq:main} having Morse index equal to one.

We are left with several open questions. It is natural to consider the following variation of the fractional Lane-Emden equation, where the nonlinearity is given by the sum of a linear and a superlinear term:
\begin{equation}\label{eq:finallinearized}
	\left\{\begin{array}{r c l l}
		(-\Delta)^s u & = & \lambda u + u^p & \text{in }\Omega, \\ u & = & 0 & \text{in }\R^N \setminus \Omega.
	\end{array}\right.
\end{equation}  
Assuming $\lambda < \lambda_1(\Omega)$, it is straightforward to prove existence of least-energy solutions to this problem. Moreover, when $\Omega$ is a ball, the results of Sections \ref{sec:linearized} and \ref{Sec:Hopf} still apply with minor modifications. Unfortunately, the proof of Proposition \ref{prop:nondegeneracy} does not carry over immediately to our case, since a condition of the form
\[ \int_B uv = 0\]
would be necessary. This condition is satisfied, by symmetry considerations, when dealing with an antisymmetric second eigenfunction, but its validity is unclear when considering radially symmetric second eigenfunctions.
	
Investigating whether \emph{positive} solutions of the fractional Lane-Emden problem in the ball are unique was the original question which motivated this paper. As we mentioned in the Introduction, in a very recent work \cite{fallwethonedim} Fall and Weth were able to prove uniqueness of positive solutions in the one-dimensional case. Extending their result to every dimension seems to be a difficult problem to be tackled.

Finally, one might wonder whether uniqueness of least-energy (or positive) solutions holds true for convex domains, as in the original result by Lin for the local case, or for other classes of domains. Difficulties arise when studying the linearized problem in order to prove nondegeneracy, since the geometry of second eigenfunctions is still not well understood in the nonlocal setting.

\subsection*{Conflict of interest:} The authors state that there is no Conflict of interest.

\subsection*{Data availability statement:} The manuscript does not have any associated data.

\end{document}